\documentclass[a4paper,10pt]{article}
\usepackage[utf8x]{inputenc}
\usepackage{amsthm,amssymb, amsmath}
\usepackage{graphicx}

\newtheorem{theorem}{Theorem}

\newtheorem{lemma}[theorem]{Lemma}
\newtheorem{proposition}[theorem]{Proposition}

\newtheorem{definition}[theorem]{Definition}
\newtheorem{remark}[theorem]{Remark}

\title{Sublinear deviation between geodesics and sample paths}
\author{Giulio Tiozzo\thanks{Department of Mathematics, Harvard University.
email: $<$tiozzo@math.harvard.edu$>$.}
}

\begin{document}
\maketitle

\begin{abstract}
We give a proof of the sublinear tracking property for sample 
paths of random walks on various groups acting on spaces with hyperbolic-like properties.
As an application, we prove sublinear tracking in Teichm\"uller distance for random walks on mapping class
groups, and on Cayley graphs of a large class of finitely generated groups. 
\end{abstract}

\section{Introduction}

In probability, the classical law of large numbers says the following: suppose $X_n$ is a sequence of 
independent, identically distributed real valued random variables, and suppose they have 
finite expectation: $\mathbb{E}[X_n] = \ell < \infty$. Then their average converges almost surely to the expectation:
\begin{equation} \label{lln}
\frac{X_1 + \dots + X_n}{n} \to \ell. 
\end{equation}

We can think of the stochastic process $Y_n := X_1 + \dots + X_n$ as a random walk on the group $\mathbb{R}$ acting by translations on the real line: 
we are starting at $x = 0$ and every time we are adding a random element $X_i$. 
In this way, the law of large numbers is equivalent to saying that almost every sample path
can be approximated by the unit speed geodesic $\gamma(t) := t$ up to an error 
which is sublinear in the number of steps: in fact, \eqref{lln} can be rewritten as
\begin{equation} \label{rtrack}
\frac{|X_1 + \dots + X_n - \gamma(\ell n) |}{n} \to 0.  
\end{equation}

As noted by Kaimanovich \cite{Ka87}, the latter formulation lends itself to a natural generalization to non-abelian groups. 
Indeed, let $G$ be a group acting isometrically on a geodesic metric space $(X, d)$, and $\mu$ a probability 
distribution on $G$.
A random walk on $G$ is defined by drawing each time independently an element $g_n$ from $G$ 
with distribution $\mu$, and considering the product 
$$w_n := g_1 \dots g_n. $$
If we fix a basepoint $x$ in $X$, the sequence $(w_n x)_{n \in \mathbb{N}}$ is a stochastic process with values in $X$, 
so we can think  of it as a random walk on $X$. In this setting, the natural question arises as a
 generalization of \eqref{rtrack} whether almost every sample path $(w_n x)$ can be approximated by some geodesic ray
$\gamma : [0, \infty) \to X$, with sublinear error:
$$\lim_{n \to \infty} \frac{d(w_n x, \gamma)}{n} = 0.$$
If such a property holds, we will say that the random walk has the \emph{sublinear tracking} or \emph{geodesic ray approximation} property.
The appropriate equivalent to finite expectation for non-abelian group actions is the \emph{finite first moment} 
condition, i.e.
$$\int_G d(x, gx) \ d\mu(g) < \infty.$$

Of particular interest in geometry and topology is the action of the \emph{mapping class group} $G = Mod(S)$ of a 
compact, orientable surface $S$ of genus $g \geq 1$ on the \emph{Teichm\"uller space} $X = \mathcal{T}(S)$, 
equipped with the \emph{Teichm\"uller metric} $d_T$. 
The first goal of this paper is to establish sublinear tracking in Teichm\"uller metric for random walks with finite 
first moment on mapping class groups:

\begin{theorem}
Let $\mu$ be a probability measure on $Mod(S)$ with finite first moment, whose support generates a non-elementary group, 
and fix some $x$ in the Teichm\"uller space $\mathcal{T}(S)$.
Then, there exists $A > 0$ such that for almost all sample paths there exists a 
Teichm\"uller geodesic ray 
$\gamma : [0, \infty) \to \mathcal{T}(S)$ with $\gamma(0) = x$ and such that 
$$\lim_{n \to \infty} \frac{d_T(w_n x, \gamma(An))}{n} = 0.$$ 
\end{theorem}

The theorem answers a question posed by Kaimanovich \cite{Ka00}. 
The theory of random products of group elements goes back to Furstenberg, who established a 
first multiplicative ergodic theorem for random walks on the group $G = GL_n(\mathbb{R})$  
(Furstenberg-Kesten \cite{FK}), then generalized to stationary, not necessarily independent increments by Oseledets \cite{O}.

In the 80's, Kaimanovich \cite{Ka87} realized that the multiplicative ergodic theorem is equivalent to sublinear 
tracking for random walks on the symmetric space $X = GL(n, \mathbb{R})/O(n, \mathbb{R})$ and proved it 
for general symmetric spaces of noncompact type. Moreover, he showed that, as a consequence of the entropy 
criterion \cite{Ka85}, sublinear tracking allows one to identify the \emph{Poisson boundary} of the random walk, 
i.e. to get a Poisson representation formula for bounded $\mu$-harmonic functions (see Theorem \ref{entrocrit}). 
This method has been applied to fundamental groups of compact Riemannian manifolds 
of negative curvature \cite{Ka85}, and word hyperbolic groups \cite{Ka94}.  

Karlsson and Margulis \cite{KarMar} proved sublinear tracking in the case of uniformly convex,
 Busemann non-positively curved spaces (which include CAT(0) spaces). Moreover, in the case of trees Ledrappier 
proved that the tracking is much better than sublinear,  namely logarithmic \cite{Le}.

It is known that Teichm\"uller space is neither Gromov hyperbolic \cite{MW} nor Busemann non-positively curved \cite{Mas75}, 
so the previous arguments do not apply; Kaimanovich and Masur \cite{KM} proved that for a random walk on the mapping class group such that the support 
of $\mu$ generates a non-elementary subgroup, almost every sample path converges to the \emph{Thurston boundary} $\mathcal{PMF}$, 
and in particular the limit foliation is almost surely uniquely ergodic. This allowed them to identify the 
Poisson boundary of such a walk with (a subset of) $\mathcal{PMF}$.
In 2005, Duchin \cite{Du} proved sublinear tracking along subsequences of times in which the limit geodesic lies 
in the thick part of Teichm\"uller space.

Let us note that Teichm\"uller space also carries the \emph{Weil-Petersson metric}, 
which is CAT(0) (though incomplete), so sublinear tracking in that metric can be proven using the 
original argument of Karlsson and Margulis \cite{KL}.

\vskip 0.3 cm 

Our method is purely ergodic theoretic, and it can be applied much more generally 
whenever one can find a compactification $\overline{X}$ of the space $X$ on which the action of $G$ extends, 
and which satisfies a few geometric properties (see Theorem \ref{mainabs}). One sufficient condition is the 
\emph{stable visibility} of the boundary: we say a compactification is \emph{stably visible} if any sequence of 
geodesics whose endpoints converge to two distinct points on the boundary intersects some bounded set of $X$ (see section \ref{tightness}).
If that property holds, we have sublinear tracking:
\begin{theorem} 
Let $G$ be a countable group acting via isometries on a proper, geodesic, metric space $(X, d)$ 
with a non-trivial, stably visible compactification $\overline{X}$. Moreover, let $\mu$ be a probability measure on 
$G$ such that the subgroup generated by $\mu$ is non-elementary, and $\mu$ has finite first moment. Then 
there exists $A \geq 0$ such that, for each $x \in X$ and for almost every sample path $(w_n x)$ there exists 
a geodesic ray $\gamma : [0, \infty) \to X$ such that 
$$\lim_{n  \to \infty} \frac{d(w_n x, \gamma(An))}{n} = 0.$$
\end{theorem}
Several interesting compactifications are stably visible, for instance: 
\begin{enumerate}
\item the \emph{hyperbolic compactification} of Gromov hyperbolic spaces (section \ref{Ghyp});
\item the \emph{end compactification} of Freudenthal and Hopf; 
\item the \emph{Floyd compactification} (section \ref{floyd}); 
\item the \emph{visual compactification} of a large class of CAT(0) spaces (section \ref{hadamard}).
\end{enumerate}
As long as these boundaries are non-trivial (i.e. they contain at least 3 points), our argument 
yields sublinear tracking in all these spaces. 

As an example, let us consider the action of a finitely generated group $G$ on its Cayley graph $X$, 
which is a geodesic space when equipped with the word metric (with respect to some choice of generators).
If $G$ is word hyperbolic, sublinear tracking in the word metric follows from considering the hyperbolic compactification
(a different proof in this case is in \cite{Ka94}). However, our method works under much weaker conditions, 
for instance as long as the \emph{Floyd boundary} is non-trivial. 
This includes the case of groups with infinitely many ends, as well as non-elementary Kleinian groups 
and relatively hyperbolic groups (see section \ref{floyd}).



Another application is to discrete groups of isometries of CAT(0) spaces, 
such as the fundamental groups of Riemannian manifolds of nonpositive sectional curvature.
In this case, a general result has been obtained in \cite{KarMar}. However, our method also works 
under some restrictions: namely, if $X$ is a Hadamard space of rank one and 
$G$ a discrete group of isometries of $X$ which satisfies the \emph{duality condition} of Eberlein-Chen (see section \ref{hadamard}).

Finally (section \ref{ll}), we shall apply our technique to random walks on lamplighter groups over trees. Lamplighter groups 
have been the center of much study because they provided several counterexamples to long-standing conjectures. 
(see especially \cite{KaiVer} and \cite{Ers}).
The Poisson boundary for lamplighter random walks over trees has been analyzed by Karlsson and Woess \cite{KarWoe}.

Note that in all the abovementioned cases our results, together with the ray criterion of Kaimanovich (Theorem \ref{entrocrit}), 
provide an identification of the Poisson boundary of the walk with a certain geometric boundary.

Note moreover that the main argument (Theorem \ref{mainabs}) extends to the case of stationary, not necessarily independent increments, i.e.
to integrable, ergodic cocycles (see Remark \ref{rem:coc}), once one assumes the almost sure convergence 
of sample paths to the boundary.

The idea of the proof is in all cases the following. Suppose for the sake of clarity that the probability measure $\mu$ has finite support, 
so the length of each step of the random walk is bounded: then, if a sample path escapes linearly from the limit geodesic, 
it takes a linear number of steps to come back close to it, hence there will be a positive frequency of times for 
which the walk is far from the geodesic. 
On the other hand, by the ergodic theorem applied to the space of bilateral sample paths, the frequency of times 
the sample path is within bounded distance from the limit geodesic is positive: since we can choose these proportions independently, 
we get a contradiction.









\medskip
\textbf{Acknowledgements.}
I wish to thank Vadim Kaimanovich, Vaibhav Gadre, Joseph Maher and Curtis McMullen for useful comments and discussions.
I especially wish to thank Anders Karlsson for many suggestions and his hospitality in Geneva in the summer of 2012.

\section{General setting}

Let $(X, d)$ be a metric space. A \emph{geodesic segment} is an isometric 
embedding of a segment $[a, b]$ into $X$, i.e. a map $\gamma : [a, b] \to X$ such that 
$d(\gamma(s), \gamma(t)) = |s-t|$ for all $s, t \in [a, b]$. A \emph{geodesic ray} is 
an isometric embedding $\gamma : [0, \infty) \to X$, while a \emph{geodesic line} is an isometric embedding
$\gamma : (-\infty, \infty) \to X$. We shall denote as $\Gamma X$ the set of geodesic lines in $X$, and as 
$\mathcal{P}(\Gamma X)$ the set of all subsets of $\Gamma X$.
A metric space is \emph{geodesic} if any two points can be joined by a geodesic segment.
It is \emph{proper} if closed balls are compact.

Let now $G$ be a countable group acting by isometries on a geodesic metric space $(X, d)$.
A \emph{bordification} 
$\overline{X}$ of $X$ will be a 
Hausdorff, second-countable 
topological space such that $X$ is homeomorphic to an open dense subset of $\overline{X}$, and such that 
the action of $G$ on $X$ extends to an action on $\overline{X}$ by homeomorphisms. 
In the following we will always identify $X$ with a subset of $\overline{X}$, and denote by $\partial X := \overline{X} \setminus X$ 
the \emph{boundary} of $X$. A bordification is \emph{non-trivial} if $\partial X$ contains at least three points.
\begin{definition}
A subgroup $G' \subseteq G$ is called \emph{elementary} with respect to the bordification $\overline{X}$ 
if it fixes a finite subset of $\partial X$; otherwise, it is called \emph{non-elementary}.
\end{definition}
Finally, a compact bordification will be called a \emph{compactification}.
Note that by our definition a compactification $\overline{X}$ is metrizable; 
moreover, the existence of a compactification according to the previous definition forces $X$ to be locally compact.

\subsection{Random walks}

Let $\mu$ be a probability measure on $G$. The \emph{step space} $\Omega:= G^\mathbb{N}$ is the space of infinite sequences
of group elements, which we will consider as a probability space with the product measure $\mathbb{P} := \mu^\mathbb{N}$.
For each $n$, let us define the $G$-valued random variable $w_n : \Omega \to G$
$$(g_1, g_2, \dots) \mapsto w_n := g_1 \dots g_n$$
given by choosing each element $g_i \in G$ independently according to the measure $\mu$, and taking the product of the 
first $n$ elements in the sequence.
Moreover, if we fix a basepoint $x \in X$, we can let $w_n$ act on $x$ hence the sequence 
$(w_n x )_{n \in \mathbb{N}}$ is a sequence of random variables on the space $(\Omega, \mathbb{P})$ 
with values in $X$ which we will call a \emph{random walk} on $X$.

A probability measure $\mu$ on $G$ is said to have \emph{finite first moment} if the average step size is finite:
$$\int_G d(x, gx) \ d\mu(g) < \infty$$
for some $x \in X$. If $\mu$ has finite first moment, then almost every sample path escapes towards 
infinity at some well-defined linear rate $A$: 

\begin{proposition}
Let $\mu$ be a probability measure on $G$ with finite first moment. Then there exists $A \geq 0$ such that 
$$\lim_{n \to \infty} \frac{d(x, w_n x)}{n} = A$$
for each $x \in X$ and $\mathbb{P}$-a.e. sample path. 
\end{proposition}

The proposition is an immediate consequence of Kingman's subadditive ergodic theorem. Note that in general $A$ can very well 
be zero, in which case the random walk remains at sublinear distance from the starting point. If instead $A$ is positive, 
we can ask whether the random walk converges ``in direction''. First of all, it makes sense to ask whether almost every sample path 
converges to some well-defined point on the boundary of $X$; moreover, since $(X, d)$ is a geodesic space, we can ask whether
there exists a geodesic ray in $X$ which approximates the sample path.
If $\mathbb{P}$-a.e. sample path converges to some point in $\partial X$, then a \emph{harmonic measure} (or \emph{hitting measure}) $\nu$ 
is defined on $\partial X$ as the pushforward of $\mathbb{P}$ with respect to the limit map:
$$\nu(A) := \mathbb{P}( \omega \in \Omega \ : \lim_{n \to \infty} w_n x \in A  ).$$
Let us denote $\check{\mu}$ the reflected measure
$$\check{\mu}(g) := \mu(g^{-1}) \qquad \forall g \in G.$$
Moreover, let $\check{\mathbb{P}}$ be the product measure $\check{\mu}^\mathbb{N}$ on $G^\mathbb{N}$, and $\check{\nu}$ 
the hitting measure on $\partial X$ relative to the reflected measure.

\medskip

Note that the measure $\nu$ is \emph{stationary}, in the sense that it satisfies the equation
$$\sum_{g \in G} \mu(g) g \nu = \nu.$$
In general, a measure space $(B, \nu)$ on which $G$ acts measurably is called a \emph{$\mu$-boundary} (or \emph{Furstenberg boundary}) 
if $\nu$ is a stationary measure and for almost every sample path $(w_n)$ the sequence of measures $(w_n \nu)$ converges to a $\delta$-measure.
The \emph{Poisson} \emph{boundary} of the random walk $(G, \mu)$ is its maximal $\mu$-boundary (it is unique up to sets of measure zero):
it can also be equivalently defined in other ways, for instance as the space of ergodic components of the shift map on 
the path space (see \cite{KaiVer} and \cite{Ka00}). 
Moreover, the Poisson boundary provides a \emph{Poisson representation formula} for harmonic functions on the group: 
indeed, if $(B, \nu)$ is a Poisson boundary of the walk $(G, \mu)$, then the formula
$$f(g) = \int_{B} \hat{f}(x)\ d(g \nu)(x)$$
provides an isomorphism between the space $L^\infty(B, \nu)$ of bounded functions on the boundary
and the space $H^\infty_\mu(G)$ of bounded harmonic functions on $G$.

By definition, the Poisson boundary is an abstract measure space;  
on the other hand, whenever the group $G$ has some geometric structure, it is often possible to construct  
a ``geometric'' boundary of $G$ (or of the space $X$ on which $G$ acts) using this structure, and a major theme of 
research is to compare this boundary to the Poisson boundary.
In this respect, an important criterion was given by Kaimanovich \cite{Ka85}, using the sublinear tracking property. 
We say that the action of $G$ on $X$ has \emph{exponentially bounded growth} if there exists $x \in X$ and 
$C \geq 0$ such that 
$$\#\{g\in G \ : \ d(x, gx) \leq R \} \leq e^{CR} \qquad \forall R > 0.$$
The following criterion implies that sublinear tracking, together with exponentially bounded growth, 
is sufficient to identify the Poisson boundary of the walk.
\begin{theorem}[\cite{Ka85}] \label{entrocrit}
Let $G$ be a countable group acting by isometries on the metric space $(X,d)$, and $\mu$ a probability measure 
on $G$ with finite first moment. Let $(B, \nu)$ be a $\mu$-boundary, and $\pi_n : B \rightarrow X$ be a sequence 
of measurable maps such that, for almost every sample path $(w_n x)$, one has
$$\lim_{n \to \infty} \frac{d(w_n x , \pi_n(\xi))}{n} = 0$$
where $\xi$ is the image of $w_n$ in $B$.  
If the action has exponentially bounded growth, then $(B, \nu)$ is the Poisson boundary of $(G, \mu)$.
\end{theorem}
A simple corollary is that if the rate of escape $A = 0$, the Poisson boundary is trivial (see also \cite{KaiVer}, \cite{De}).


\subsection{Abstract sublinear tracking} \label{abstrack}

The goal of this section is to prove an abstract sublinear tracking criterion. 
Roughly speaking, the only two ingredients that are needed are that almost every sample path converges to the boundary, and that 
given two different points of the boundary, it is possible to choose a geodesic line \emph{in a G-equivariant way}.

\begin{theorem}\label{mainabs}
Let $G$ be a countable group acting by isometries on a geodesic metric space $X$, and let $\overline{X} = X \cup \partial X$
be a bordification of $X$. 
Let $\mu$ be a probability measure on $G$ with finite first moment, and suppose the following are true:
\begin{enumerate}
\item $\mathbb{P}$-a.e. sample path $(w_n x)$ converges to some  $\xi \in \partial X$, 
and $\check{\mathbb{P}}$-a.e. sample path $(\check{w}_n x)$ converges to some  $\eta \in \partial X$.
\item There exists a $G$-equivariant map $P : \partial X \times \partial X \to \mathcal{P}(\Gamma X)$
which associates to any pair of points of the boundary a set of geodesics in $X$ in such a way that the map 
$D : \partial X \times \partial X \to \mathbb{R}$ defined as 
$$D(\eta, \xi) := \sup_{\gamma \in P(\eta, \xi)} d(x, \gamma)$$
is Borel-measurable 
 and finite $\nu \otimes \check{\nu}$-a.e. (note that we set $D(\eta, \xi) = \infty$ if $P(\eta, \xi) = \emptyset$, so the condition includes that $P(\eta, \xi)$ is almost surely non-empty).
\end{enumerate}
Then, there exists $A \geq 0$ such that for $\mathbb{P}$-a.e. sample path $(w_n x)$ there exists a geodesic ray $\gamma :[0, \infty) \to X$ such that 
$$\lim_{n \to \infty} \frac{d(w_n x, \gamma(An))}{n} = 0.$$
\end{theorem}

The proof of the theorem is based on the following elementary lemma in ergodic theory:

\begin{lemma}
Let $\Omega$ be a measure space with a probability measure $\lambda$, and let $T : \Omega \to \Omega$ be a measure-preserving, 
ergodic transformation.
Let $f : \Omega \to \mathbb{R}$ be a non-negative, measurable function, and define the function $g : \Omega \to \mathbb{R}$ as
\begin{equation} \label{cc}
g(\omega) := f(T \omega) - f(\omega) \qquad \forall \omega \in \Omega.
\end{equation}
If $g \in L^1(\Omega, \lambda)$, then, for $\lambda$-almost every $\omega \in \Omega$, one has 
$$\lim_{n \to \infty} \frac{f(T^n \omega)}{n} = 0.$$
\end{lemma}

\begin{proof}
As a consequence of \eqref{cc} we can write, for each $n \geq 1$ and each $\omega \in \Omega$,
$$\frac{\sum_{k = 0}^{n-1} g(T^k \omega)}{n} = \frac{f(T^n \omega) - f(\omega)}{n}.$$
By the ergodic theorem, for almost every $\omega$, the left-hand side converges to $\int_\Omega g \ d\lambda$.
This implies that for $\lambda$-almost every $\omega$, the limit
$$\lim_{n \to \infty}  \frac{f(T^n \omega)}{n}$$
exists.
On the other hand, there exists $C >0$ such that $\Omega_C := \{ \omega \ : \ f(\omega) \leq C \}$
has positive measure $\mu(\Omega_C) > 0$.
Again by the ergodic theorem, for $\lambda$-almost every $\omega$, the set 
$$\{ n \in \mathbb{N} \ : \ f(T^n \omega) \leq C \}$$
has positive density, 
which implies that
$$\liminf_{n \to \infty} \frac{f(T^n \omega)}{n} \leq \liminf_{n \to \infty} \frac{C}{n} = 0.$$
As a consequence, since the limit exists, we have
$$\lim_{n \to \infty} \frac{f(T^n \omega)}{n} =  \liminf_{n \to \infty} \frac{f(T^n \omega)}{n} = 0.$$

\end{proof}

\begin{proof}[Proof of Theorem \ref{mainabs}]
Let us apply the lemma with $\Omega := G^\mathbb{Z}$ the space of all bi-infinite sequences, 
endowed with the product measure $\lambda := \mu^\mathbb{Z}$.
Let $x \in X$ be a fixed base point, and define the boundary maps $\textup{bnd}^\pm : G^\mathbb{Z} \to \partial X$ as
\begin{equation} 
\textup{bnd}^+(\underline{g}) := \lim_{n \to \infty}  g_1 \cdots g_n x
\qquad 
\textup{bnd}^-(\underline{g}) := \lim_{n \to \infty} g_0^{-1}g_{-1}^{-1} \cdots g_{-n}^{-1} x.
\end{equation}
Note that by condition 1. $\textup{bnd}^\pm$ are defined for almost every sequence in $G^\mathbb{Z}$.
Let us now take $T := \sigma$ the shift on the space of bi-infinite sequences, which acts ergodically on $G^\mathbb{Z}$.
Note that, for each $\underline{g} \in G^\mathbb{Z}$ (recall $w_n = g_1 \cdots g_n$), 
\begin{equation} \label{bndequiv}
\textup{bnd}^+(\sigma^n \underline{g}) = w_n^{-1} \textup{bnd}^+(\underline{g}) \qquad \textup{bnd}^-(\sigma^n \underline{g}) = 
w_n^{-1} \textup{bnd}^-( \underline{g}).
\end{equation}
We are now ready to define the non-negative function $f : G^\mathbb{Z} \to \mathbb{R}$ as 
the maximum distance between the base point and any geodesic joining the two limits of the random walk given by 
a fixed sequence:
$$f(\underline{g}) := \sup_{\gamma \in P(\textup{bnd}^+(\underline{g}), \textup{bnd}^-(\underline{g}))} d(x, \gamma).$$
By condition 2., $f$ is finite for almost all sequences. Let us now see how the shift $\sigma$ acts on $f$: 
by definition one has 
$$f(\sigma^n \underline{g}) =  \sup_{\gamma \in P(\textup{bnd}^+(\sigma^n \underline{g}), \textup{bnd}^-(\sigma^n \underline{g}))} d(x, \gamma)$$
and, since $G$ acts by isometries,
$$f(\sigma^n \underline{g})   = \sup_{w_n \gamma \in P(w_n \textup{bnd}^+(\sigma^n \underline{g}), w_n \textup{bnd}^-(\sigma^n \underline{g}))} d(w_n x, w_n \gamma)$$
then, by eq. \eqref{bndequiv} we have
$$  f(\sigma^n \underline{g}) = 
\sup_{\gamma \in P(\textup{bnd}^+(\underline{g}), \textup{bnd}^-(\underline{g}))} d(w_n x,\gamma). $$
Now, by triangle inequality we have
\begin{equation} \label{triangle}
|f(\sigma \underline{g}) - f(\underline{g})| \leq d(x, g_0 x)
\end{equation}
hence the finite first moment implies that the function $F(\underline{g}) := f(\sigma \underline{g}) - f(\underline{g})$ is integrable, so one can apply the lemma and get, 
for a.e. $\underline{g} \in G^\mathbb{Z}$, 
$$\lim_{n \to \infty} \frac{\sup_{ \gamma \in P(\textup{bnd}^+(\underline{g}), \textup{bnd}^-(\underline{g})) } d(w_n x, \gamma)}{n} = 0.$$
This obviously implies, for each $\gamma \in P(\textup{bnd}^+(\underline{g}), \textup{bnd}^-(\underline{g}))$, 
$$\lim_{n \to \infty} \frac{ d(w_n x, \gamma)}{n} = 0.$$
Let us now choose some $\gamma \in P(\textup{bnd}^+(\underline{g}), \textup{bnd}^-(\underline{g}))$ and fix a parametrization 
$\gamma : \mathbb{R} \to X$. The previous equation implies the existence of a sequence $(t_n)$ of times such that
$$\lim_{n \to \infty} \frac{ d(w_n x, \gamma(t_n))}{n} = 0.$$
Notice that this implies $\lim_{n \to \infty}\frac{|t_n|}{n} = A$; now, by finite first moment 
for a.e. sample path 
 $\frac{d(w_n x, w_{n+1} x)}{n} \to 0$, hence either $\lim_{n \to \infty} \frac{t_n}{n} = A$ or 
$\lim_{n \to \infty} \frac{t_n}{n} = -A$. In the first case, $w_n x$ is approximated by the positive ray $\gamma\mid_{[0, \infty)}$, 
in the second by the negative ray $\gamma\mid_{(-\infty, 0]}$. 
\end{proof}

Let us observe that the proof does not require the space $\partial X$ to be compact, 
so it can be applied to non-proper spaces as long as one can prove convergence to the boundary.
Moreover, we intuitively think of the map $P : \partial X \times \partial X \to \Gamma X$ as choosing
a geodesic which ``joins'' two points of the boundary, but actually the only property we need 
is that the choice is $G$-equivariant. In particular, we do not require that, if the geodesic $\gamma$
belongs to $P(\eta, \xi)$, $\gamma(t)$ tends to $\eta$ (or $\xi$) as $t \to \infty$.

\begin{remark} \label{rem:coc}
We also never use that the increments of our walk are independent, so it works more generally
in the context of cocycles. Indeed, let $u : \Omega \to G$ be a measurable map defined on some probability space $(\Omega, \lambda)$, 
and suppose $T : \Omega \to \Omega$ is an ergodic, measure preserving map. We can define the cocycle 
$u : \mathbb{N} \times \Omega \to G$ as 
$$u(n, \omega) := u(\omega) u(T\omega) \cdots u(T^{n-1} \omega).$$
The cocycle is \emph{integrable} if $\int_\Omega d(x, u(n, \omega)x)\ d\lambda(\omega) < \infty$ for some $x \in X$. 
The previous argument yields sublinear tracking for integrable, ergodic cocycles, once again under the hypothesis 
of convergence to the boundary.
\end{remark}


 

\subsection{Stable visibility} \label{tightness}

Let us now formulate a geometric property of a compactification that, 
at least in the case of proper spaces, is sufficient to yield sublinear tracking. 
We call a compactification $\overline{X}$ \emph{stably visible} if any sequence of geodesic segments 
whose endpoints converge to two distinct points on the boundary intersects some bounded set of $X$:

\begin{definition}  \label{svdef}
A compactification $\overline{X}$ of a geodesic metric space $(X, d)$ is \emph{stably visible} if the following holds:
given any sequence $\gamma_n = [\eta_n, \xi_n]$ of geodesic segments in $X$ connecting $\eta_n$ with $\xi_n$ 
and such that $\xi_n \to \xi \in \partial X$, $\eta_n \to \eta \in \partial X$ with $\eta \neq \xi$, there exists 
a bounded set $B$ in $X$ which intersects all geodesics $\gamma_n$.
\end{definition}

The goal of this section is to prove the following theorem:

\begin{theorem} \label{tighttracking}
Let $G$ be a countable group acting via isometries on a proper, geodesic, metric space $(X, d)$ 
with a non-trivial, stably visible compactification $\overline{X}$. Moreover, let $\mu$ be a probability measure on 
$G$ such that the subgroup generated by $\mu$ is non-elementary, and $\mu$ has finite first moment. Then 
there exists $A \geq 0$ such that, for each $x \in X$ and for almost every sample path $(w_n x)$ there exists 
a geodesic ray $\gamma : [0, \infty) \to X$ such that 
$$\lim_{n  \to \infty} \frac{d(w_n x, \gamma(An))}{n} = 0.$$
Moreover, if the action has exponentially bounded growth, then $A > 0$.
\end{theorem}

Note that in order to apply Theorem \ref{mainabs} it is necessary to produce a map from $\partial X \times \partial X$ to 
subsets of the set of geodesics $\Gamma X$. In this section, 
given a pair of points $\xi, \eta$ on $\partial X$, we will denote as $P(\xi, \eta)$ the set of geodesic lines $\gamma : \mathbb{R} \to X$ 
such that $\lim_{t \to \infty} \gamma(t) = \xi$ and $\lim_{t \to -\infty} \gamma(t) = \eta$. Such a set will be called a \emph{pencil}.

\begin{lemma} \label{tightprop}
Let $X$ be a proper, geodesic, metric space with a stably visible compactification $\overline{X}$. Then:
\begin{enumerate}
\item the action of $G$ is \emph{projective}: if $g_n x \to \xi \in \partial X$ for some $x \in X$, then $g_n y \to \xi$ for all $y \in X$;
\item as $X$ is proper, then for each $\eta, \xi \in \partial X$ with $\eta \neq \xi$, the pencil $P(\xi, \eta)$ is non-empty;
\item for each pair $(\eta, \xi)$ with $\eta \neq \xi$, there exists a bounded set $B \subseteq X$ 
which intersect all geodesics in $P(\eta, \xi)$;
\item if $\xi_0$, $\xi_1$ and $\xi_2$ are three distinct points of $\partial X$, then there are neighbourhoods $U_0$, $U_1$, $U_2$  
of $\xi_0$, $\xi_1$, $\xi_2$ in $\overline{X}$ such that each geodesic $\gamma$ joining $\eta_1 \in U_1$ and $\eta_2 \in U_2$ is disjoint from $U_0$.

\end{enumerate}

\end{lemma}

\begin{proof}
1. Suppose $g_n x \to \xi \in \partial X$. By properness, $d(g_n x, x)$ is unbounded, and moreover let us note that 
 the distance $d(g_n x, g_n y) = d(x, y)$ is bounded independently of $n$.
Suppose now that there is a (sub)sequence $g_n$ such that $g_n y \to \eta \in \overline{X}$, $\eta \neq \xi$.
Since $g_n y$ is also unbounded in $X$, then $\eta \in \partial X$.   
Now, let us choose for each $n$ a geodesic $\gamma_n$
joining $g_nx$ and $g_n y$. Then by stable visibility there exists a ball $B \subseteq X$ 
which interesects all $\gamma_n$. Since the distance $d(g_n x, g_n y)$ is bounded, the union 
of all  $\gamma_n$ must lie in some ball $B' \subseteq X$, which contradicts the unboundedness of $g_n x$.

2. Let us take a sequence $\xi_n \to \xi$ and $\eta_n \to \eta$, and geodesics $\gamma_n$ joining them. The claim 
follows by the Ascoli-Arzel\'a theorem.

3. Immediate.

4. If the claim is false, then there exist sequences $\xi_{0,n}$, $\xi_{1, n}$ and $\xi_{2, n}$ 
such that $\xi_{i, n} \to \xi_i$ for each $i = 0, 1, 2$, and geodesics $\gamma_n$ 
which join $\xi_{1, n}$ and $\xi_{2, n}$ and such that $\xi_{0, n}$ also belongs to $\gamma_n$:
let us denote as $\gamma_{1, n}$ the part of $\gamma_n$ between $\xi_{1, n}$ and $\xi_{0, n}$, and 
as $\gamma_{2, n}$ the part between $\xi_{0, n}$ and $\xi_{2, n}$.
Then by stable visibility, there exists a ball $B \subseteq X$ which intersects all geodesics $\gamma_{1, n}$ 
and $\gamma_{2, n}$: let us pick a point $\alpha_n \in B \cap \gamma_{1, n}$ and $\beta_n \in B \cap \gamma_{2, n}$.
Then clearly the distance between $\alpha_n$ and $\beta_n$ is bounded, whereas $d(\alpha_n, \xi_{0, n}) \to \infty$
and $d(\beta_n, \xi_{0, n}) \to \infty$, contradicting the fact that $\alpha_n, \xi_{0,n}$ and $\beta_n$ 
lie in that order on the geodesic $\gamma_n$. 

\end{proof}

\begin{lemma} \label{measurable}
Let $X$ be a proper geodesic space with a stably visible compactification. Then for each $x \in X$ 
the function $D : \partial X \times \partial X \to \mathbb{R}$
$$D(\xi, \eta) := \sup_{\gamma \in P(\xi, \eta)} d(x, \gamma)$$
is measurable with respect to the Borel $\sigma$-algebra.
\end{lemma}

\begin{proof}
Recall $\overline{X}$ is a second-countable compact Hausdorff space, hence metrizable, and $\partial X$ is a closed subset, hence compact. 
Pick a metric $\tilde{d}$ on $\overline{X}$, and for each integer $k \geq 1$ pick a cover of $\partial X$ made of finitely many $\tilde{d}$-balls of 
radius $\frac{1}{k}$ and centers on $\partial X$.
Taking the union over all $k$ yields a countable sequence $( U_n )_{n \in \mathbb{N}}$ of open sets in $\overline{X}$ with the following properties:

\begin{enumerate} 
\item the $U_n \cap \partial X$ are a base for the topology of $\partial X$;
\item for each $R$, only finitely many $U_n$ intersect the ball $B(x, R)$ in $X$;
\item for each sequence $n_k \to \infty$, the intersection $\bigcap_{k=1}^\infty U_{n_k}$ contains at most one point.
\end{enumerate}
Let us now fix some $R > 0$, and say that a pair $(U, V)$ of open sets in $\overline{X}$ \emph{avoids the ball of radius $R$} if 
there is a point $u \in U \cap X$, a point $v \in V \cap X$ and a geodesic segment $\gamma$ 
joining $u$ to $v$ which does not intersect the ball $B(x, R)$.  Let us define the collection
$\mathcal{S} := \{ (U_n, U_m) \ : \ (U_n, U_m) \textup{ avoids the ball of radius }R \}$
which is a countable collection of pairs of open sets. 
We claim that the following identity holds
$$\{ (\eta, \xi) \in \partial X \times \partial X \ : \ D(\eta, \xi) \geq R \} = \ \bigcap_{N} \bigcup_{\stackrel{\min\{m, n \} \geq N}{(U_n, U_m) \in \mathcal{S}} } U_n \times U_m$$
which implies $D$ is measurable.
Indeed, one inclusion is trivial: if $\xi$ and $\eta$ are joined by a geodesic avoiding the ball of radius $R$, however we choose neighbourhoods $U$ and $V$ of $\xi$ and $\eta$ respectively, 
then the pair $(U, V)$ avoids the ball of radius $R$. On the other hand, suppose $(\eta, \xi)$ belongs 
to the intersection of a sequence $U_{a_k} \times U_{b_k}$ 
of pairs of open sets avoiding the ball of radius $R$: 
by definition, there is a sequence $\gamma_k$ of geodesics joining a point $\xi_k$ of $U_{a_k}$ to a point $\eta_k$ of $U_{b_k}$
and avoiding the ball of radius $R$. By stable visibility, all geodesics $\gamma_k$ must intersect some ball $B(x, R')$, hence 
by properness and the Ascoli-Arzel\'a theorem there exists a geodesic line $\gamma$ which joins $\xi$ and $\eta$, avoiding $B(x, R)$.
\end{proof}

\begin{proof}[Proof of Theorem \ref{tighttracking}]
Conditions 1. and 4. of Lemma \ref{tightprop} correspond to conditions (CP) and (CS) of Kaimanovich \cite{Ka00}. As a consequence of Theorem 2.4 in 
\cite{Ka00}, if the compactification is non-trivial and the group generated by the support of $\mu$ is non-elementary, then $\mathbb{P}$-a.e. 
sample path converges to a point in $\partial X$, and so does $\check{\mathbb{P}}$-a.e. backward sample path. Moreover, the limit measures 
$\nu$ and $\check{\nu}$ are non-atomic (hence the diagonal in $\partial X \times \partial X$ has zero measure). Moreover, 
by Lemma \ref{tightprop}.3 and Lemma \ref{measurable} the function $D(\xi, \eta) := \sup_{\gamma \in P(\xi, \eta)} d(x, \gamma)$, 
where $P(\eta, \xi)$ is the pencil of geodesics joining $\eta$ and $\xi$, is a.e. finite and measurable, hence we can apply Theorem \ref{mainabs}
and get the main claim. Finally, since $\nu$ is non-atomic, the Poisson boundary of the walk is non-trivial. 
As a consequence, if the action has exponentially bounded growth, then by the entropy criterion (Theorem \ref{entrocrit})
the rate of escape $A$ cannot be zero.
\end{proof}








\section{Applications}

\subsection{Gromov hyperbolic spaces} \label{Ghyp}

The first setting where we apply our technique is in Gromov hyperbolic spaces: we treat it first mainly 
because of its simplicity.
Let $X$ be a geodesic metric space, and fix $\delta \geq 0$. Let us recall that $X$ is called \emph{ $\delta$-hyperbolic}
if geodesic triangles are $\delta$-thin, which means that given any three points $x, y, z \in X$, the geodesic $[x, y]$ 
lies in a $\delta$-neighbourhood of the union $[x, z] \cup [z, y]$.
The \emph{Gromov product} of $y$ and $z$ with respect to $x$ is defined as 
$$(y, z)_x := \frac{1}{2}(d(x, y) + d(x, z) - d(y, z)).$$

A $\delta$-hyperbolic space is naturally endowed with the \emph{hyperbolic boundary} $\partial X$: 
\begin{definition} 
The \emph{hyperbolic boundary} $\partial X$ of $X$ is 
the set of sequences $(x_n) \subseteq X$ such that 
$$\liminf_{n, m \to \infty} (x_n, x_m)_x = \infty$$
modulo the equivalence relation $ (x_n ) \sim (y_n)$
if
$$ \liminf_{m,n \to \infty} (x_n, y_m)_x = \infty.$$
\end{definition}

Now, if $X$ is proper, then $\overline{X} := X \cup \partial X$ can be given a second-countable, Hausdorff topology
in such a way that $\overline{X}$ is compact (we refer to \cite{BH} for background material). Moreover, $\partial X$ coincides 
with the \emph{visual compactification} $\partial_v X$
given by equivalence classes of geodesic rays, where two rays are identified if their distance stays bounded:
$$\partial_v X := \{ \gamma : [0, \infty) \to X \textup{ geodesic ray } \}/\sim$$
with $\gamma_1 \sim \gamma_2$ if $\sup_t d(\gamma_1(t), \gamma_2(t)) < \infty$. Moreover, the action of $G$ by isometries 
extends to an action on $\overline{X}$ by homeomorphisms. Let us check stable visibility:

\begin{lemma}
The hyperbolic compactification of a proper, $\delta$-hyperbolic space $X$ is stably visible.
\end{lemma}

\begin{proof}
Pick two sequences $\xi_n \to \xi$ and $\eta_n \to \eta$. Since $\eta \neq \xi$, the Gromov product $(\xi_n, \eta_n)_x$ 
is bounded. The claim now follows from the 
fact that in a $\delta$-hyperbolic space the Gromov product approximates the distance to the geodesic, namely 
for any three points $x, y, z \in X$ the following inequality holds:
$$(y, z)_x \leq d(x, [y, z]) \leq (y, z)_x + 2 \delta$$
where $[y,z]$ is a geodesic joining $y$ and $z$.
\end{proof}

By the results of the previous section we can thus infer the following sublinear tracking:

\begin{theorem}
Let $G$ be a countable group of isometries of a proper, geodesic $\delta$-hyperbolic space $(X, d)$, 
such that the hyperbolic  boundary $\partial X$ contains at least three points. Moreover, 
let $\mu$ be a probability measure on $G$ with finite 
first moment and such that the group generated by its support is non-elementary. 
Then there exists $A \geq 0$ such that for any $x \in X$, almost every sample path $(w_n x)$ 
converges to some point in $\partial X$, and there exists a geodesic ray $\gamma : [0, \infty) \to X$ with 
$\gamma(0) = x$ such that 
$$\lim_{n \to \infty} \frac{d(w_n x, \gamma(An))}{n} = 0.$$ 
\end{theorem}

\begin{proof}
By Theorem \ref{tighttracking}, for almost every sample path $(w_n x)$  there exist a geodesic ray $\gamma : [0, \infty) \to X$
which tracks the sample path sublinearly. Such a geodesic need not pass through $x$: however, since the hyperbolic boundary 
and the visual boundary coincide, there exists a geodesic ray $\gamma'$ such that $\gamma'(0) = x$ and 
$\lim_{t \to \infty} \gamma'(t) = \eta = \lim_{t \to \infty} \gamma(t)$ and such geodesic lies within bounded 
Hausdorff distance of $\gamma$, yielding the result. 
\end{proof}

A different proof of sublinear tracking on word hyperbolic groups is already contained in \cite{Ka94} (and attributed to T. Delzant), 
where it is used to prove that the hyperbolic compactification coincides with the Poisson boundary of the walk.

\subsection{Groups with non-trivial Floyd boundary} \label{floyd}

Let $G$ be a finitely generated group, and let us denote as $\Gamma = \Gamma(G, S)$ the Cayley graph 
relative to a generating set $S$. $\Gamma$ is a proper, geodesic metric space with respect to the word metric 
given by $S$. Several compactifications of $\Gamma$ have been studied, 
starting with the \emph{end compactification} $\mathcal{E}(G)$ of Freudenthal and Hopf \cite{Hopf}. 
It is not hard to check that the end compactification is stably visible according to the definition of section \ref{tightness}, 
hence we can apply theorem \ref{tighttracking} and get sublinear tracking for random walks on the Cayley 
graph, as long as the group generated by the support of $\mu$ is non-elementary with respect to $\mathcal{E}(G)$
(convergence to the boundary for groups with infinitely many ends is due to Woess \cite{Woe93}).

However, several interesting groups (for instance fundamental groups of compact hyperbolic surfaces) turn out to have 
trivial end compactification. 
In 1980, W. Floyd \cite{Fl} introduced a finer compactification of the Cayley graph of a finitely generated group, 
and applied it to the study of Kleinian groups. Let us now recall Floyd's construction.

Let $F$ be a summable, decreasing function $F : \mathbb{N} \to \mathbb{R}^+$ such that given $k \in \mathbb{N}$ 
there exist $M, N > 0$ so that 
$$MF(r) \leq F(kr) \leq NF(r) \qquad \forall r.$$
Let us now define a new metric on $\Gamma(G, S)$: namely, let us set the length of the edge between vertices $a, b \in G$
to be $\min\{ F(|a|), F(|b|) \}$, and let us extend it to a metric $d_F$ on $\Gamma$ by taking shortest paths.
The \emph{Floyd compactification} is the completion of $\Gamma(G, S)$ with respect to $d_F$. The complement 
of $\Gamma$ in the completion will be called a \emph{Floyd boundary} $\partial_F G$ (note it depends 
on the choice of $F$). A Floyd boundary is called \emph{non-trivial} if it contains at least three points 
(which implies it must contain infinitely many of them). As long as such boundary is non-trivial, we have sublinear tracking:

\begin{theorem} \label{floydtracking}
Let $G$ be a finitely generated group, $S$ a generating set and let $\Gamma= \Gamma(G, S)$ be
its Cayley graph, with the associated word metric $d$. Suppose $G$ has non-trivial Floyd boundary $\partial_F G$
with respect to some scaling function $F$, 
and let $\mu$ be a probability measure on $G$
with finite first moment and such that the group generated by the support of $\mu$ is non-elementary. 
Then there exists $A > 0$ such that 
for almost every sample path $(w_n)$ there exists a geodesic ray $\gamma : [0, \infty) \to \Gamma$
such that 
$$\lim_{n \to \infty} \frac{d(w_n, \gamma(An))}{n} = 0.$$
\end{theorem}

The proof of the theorem is based on the fact that the Floyd compactification is stably visible, which follows 
from the following lemma of A. Karlsson \cite{Kar03} (which is used to identify the Floyd and Poisson boundaries):

\begin{lemma} \label{floydlemma}
Let $z$ and $w$ be two points in $\Gamma$ and let $[z, w]$ be a geodesic segment connecting $z$ and $w$. Then 
$$d_F(z, w) \leq 4 r F(r) + 2 \sum_{j = r}^\infty F(j)$$
where $r = d(e, [z, w])$.
\end{lemma}

\begin{proof}[Proof of Theorem \ref{floydtracking}]
It is enough to check stable visibility: indeed, let $\xi_n \to \xi \in \partial_F X$ and $\eta_n \to \eta \in \partial_F X$, with $\xi \neq \eta$. 
Then $d_F(\eta, \xi ) > 0$, hence by lemma \ref{floydlemma} and summability of $F$, the set of values $r_n := d(e, [\xi_n, \eta_n])$ is 
bounded. The claim now follows by Theorem \ref{tighttracking} (note that the action has exponentially bounded growth
because $G$ is finitely generated).
\end{proof}

Since there is a natural surjection from the Floyd boundary onto the space of ends $\mathcal{E}(G)$, groups with infinitely many ends 
 also have non-trivial Floyd boundary, hence the previous theorem applies. 

The theorem also applies to Kleinian groups: indeed, Floyd \cite{Fl} constructed a continuous 
surjection from the boundary of a geometrically finite Kleinian group onto its limit set, 
and more recently the same result was extended by Mahan Mj \cite{MM} to arbitrary finitely generated Kleinian groups.
As a consequence, Theorem \ref{floydtracking} yields sublinear tracking for random walks on a non-elementary, 
finitely generated Kleinian group.

Finally, Gerasimov \cite{Ge} recently constructed the Floyd map for general relatively hyperbolic groups. Namely, 
he proved for any relatively hyperbolic group $G$ that there exists a surjection from some Floyd boundary of $G$ 
(defined using an exponential scaling function) onto the Bowditch boundary. Thus, non-elementary relatively hyperbolic groups 
have non-trivial Floyd boundary, hence we can apply our theorem and obtain sublinear tracking in the word metric 
on $G$.

\subsection{Teichm\"uller space}

Let now $S$ be a closed surface of genus $g \geq 1$, and $X = \mathcal{T}(S)$ the Teichm\"uller space 
of $S$, endowed with the \emph{Teichm\"uller metric} $d_T$.  The \emph{mapping class group} 
$$Mod(S) := \textup{Diff}^+(S)/\textup{Diff}_0(S)$$
of orientation-preserving diffeomorphisms of $S$ modulo isotopy acts on $\mathcal{T}(S)$, and the quotient is the 
\emph{moduli space} of Riemann surfaces of genus $g$. 

A vast area of research has addressed the question of what properties $\mathcal{T}(S)$ shares with spaces of negative curvature.
Masur \cite{Mas75} showed that Teichm\"uller space is not a CAT(0) space, and not even Busemann non-positively curved.
Moreover, it is not Gromov hyperbolic (Masur and Wolf \cite{MW}); Minsky \cite{Mi} also proved that near the cusp Teichm\"uller metric 
can be modeled on a sup metric of a product of lower dimensional spaces, which is also in contrast 
with negative curvature geometry. For a survey on the geometric properties of the Teichm\"uller metric, we refer to \cite{Mas09}.

In terms of random walks on $Mod(S)$, Kaimanovich and Masur \cite{KM} proved that almost every sample path converges
to some point in the Thurston compactification, and identified the Poisson boundary of the walk with the set $\mathcal{UE}$ 
of uniquely ergodic projective measured foliations. Duchin \cite{Du} proved that, for walks with finite first moment, the random walk tracks 
Teichm\"uller geodesics sublinearly along subsequences of times for which the geodesic lies in the thick part of moduli space. 

Our technique yields sublinear tracking for random walks with finite first moment without any restriction:

\begin{theorem} \label{teichtracking}
Let $\mu$ be a distribution on $Mod(S)$ with finite first moment, whose support generates a non-elementary group.
Then, there exists $A > 0$ such that for each $x \in \mathcal{T}(S)$ and for almost all sample paths $(w_n x)$ 
there exists a Teichm\"uller geodesic ray $\gamma$ which passes through $x$ and such that 
$$\lim_{n \to \infty} \frac{d_T(w_n x, \gamma(An))}{n} = 0.$$ 
\end{theorem}

Let us also remark that in \cite{Du}, the tracking is a consequence of an additional geometric property 
of Teichm\"uller space (``thin-framed triangles are thin''). Our result is completely independent of it, and indeed uses only \cite{KM}.

In order to prove Theorem \ref{teichtracking}, we are going to use the \emph{Thurston compactification} of Teichm\"uller space. 
Let $\mathcal{S}$ be the set of homotopy classes of simple closed curves. The \emph{geometric intersection number}
$i(\alpha, \beta)$ between two elements of $\mathcal{S}$ is the minimal number of intersections of any two representatives 
of $\alpha$ and $\beta$. The map 
$$\alpha \to i(\cdot, \alpha)$$
defines an inclusion of $\mathcal{S}$ into the set $\mathbb{R}^\mathcal{S}$ of functions on the set of simple closed curves.
The closure of the image of the set $\{ r\alpha, \alpha \in \mathcal{S}, r \geq 0\}$ is the space $\mathcal{MF}$ of \emph{measured foliations},
and its projectivization is denoted as $\mathcal{PMF}$, the space of \emph{projective measured foliations}.
The intersection number $i( \cdot, \cdot)$ extends to a continuous function on $\mathcal{MF} \times \mathcal{MF}$, 
and given two elements $F_1, F_2$ in $\mathcal{PMF}$ it is well-defined whether $i(F_1, F_2)$ is zero or non-zero.

As discovered by Thurston, the space $\overline{X} = \mathcal{T} \cup \mathcal{PMF}$ can be given a topology
which makes it homeomorphic to a closed ball in euclidean space, in such a way that $\mathcal{PMF}$ 
corresponds to the boundary sphere. The space $\mathcal{PMF}$ is called the \emph{Thurston boundary} of Teichm\"uller space, 
and the mapping class group acts on it by homeomorphisms.

A measured foliation $F$ is \emph{minimal} if it intersects all simple closed curves, i.e. $i(F, \alpha) > 0$ 
for all $\alpha \in \mathcal{S}$.
Two minimal foliations $F$ and $G$ are said to be \emph{topologically equivalent} if $i(F, G) = 0$, and \emph{transverse} 
if $i(F, G) > 0$. If $F$ is minimal  and the only topologically equivalent foliations to $F$ are its multiples, then $F$ is said to be \emph{uniquely ergodic}. 
The space of projective classes of (minimal) uniquely ergodic measured foliations will be denoted by $\mathcal{UE}$.

In the previous cases, in order to prove the theorem we used the stable visibility of the compactification to associate to 
almost each pair of points on the boundary a geodesic in an equivariant way. In this case, even though stable visibility fails if the genus of $S$ is at least two, 
there is a more direct way to construct such a function. 
Indeed, let $Q(S)$ be the bundle of \emph{quadratic differentials} over Teichm\"uller space. The choice of some $q \in Q(S)$ determines 
a flat (often singular) structure on $S$, and a pair of transverse measured foliations, namely the  
\emph{horizontal} foliation $H_q := \textup{ker}(\textup{Re }\ q^{1/2})$ and the \emph{vertical} foliation $V_q := \textup{ker}(\textup{Im} \ q^{1/2})$. 

On the other hand, given any two transverse measured foliations $F_1, F_2 \in \mathcal{MF}$, there exists a unique quadratic differential 
$q \in Q(S)$ such that $H_q = F_1$ and $V_q = F_2$.

The Teichm\"uller geodesic flow just expands along the leaves of the horizontal foliation and shrinks along the vertical foliation by the same factor, hence 
any pair $F_1, F_2 \in \mathcal{PMF}$ of transverse projective measured foliations determines a unique Teichm\"uller geodesic: 
we will denote such geodesic as $[F_1, F_2]$.

\begin{lemma}
Given $x \in \mathcal{T}(S)$, the function $D: \mathcal{PMF} \times \mathcal{PMF} \to \mathbb{R}$ 
$$D(F_1, F_2) = d_T(x, [F_1, F_2])$$ 
is defined $\nu \otimes \check{\nu}$-a.e. and measurable. 
\end{lemma}

\begin{proof}
By (\cite{KM}, Theorem 2.2.4), $\nu$-a.e. sample path converges to some uniquely ergodic projective measured foliation $F \in \mathcal{UE} \subseteq \mathcal{PMF}$, 
and so does $\check{\nu}$-a.e. backward sample path. Moreover, since the measures $\nu$ and $\check{\nu}$ are non-atomic, the diagonal  
$\Delta := \{ (F, F) \ : \ F \in \mathcal{UE} \}$ has zero measure. Since two non-equivalent minimal uniquely ergodic foliations 
are transverse, the function $D$ is defined on the set $\mathcal{UE} \times \mathcal{UE} \setminus \Delta$, which has full measure. Moreover, $D$ is continuous on the subset of $\mathcal{PMF} \times 
\mathcal{PMF}$ where it is defined (\cite{KM}, Lemma 1.4.3) hence the measurability.
\end{proof}

\begin{proof}[Proof of Theorem \ref{teichtracking}.]
Kaimanovich and Masur \cite{KM} proved that almost every sample path $(w_n x)$ converges almost surely to some 
uniquely ergodic projective measured foliation $F$. By Theorem \ref{mainabs} and the previous lemma, 
$(w_n x)$ tracks sublinearly some geodesic ray $\gamma : [0, \infty) \to \mathcal{T}(S)$. 
Moreover, the rate of escape is positive because the Poisson boundary of $(G, \mu)$ is non-trivial, 
and the action has exponentially bounded growth (\cite{KM}, Theorem 1.3.2 and Corollaries). 
Let now $y := \gamma(0)$, and $\gamma'$ be a geodesic through $y$ determined by the 
same vertical foliation as $\gamma$. 
Now, $\gamma$ and $\gamma'$ are geodesic rays with the same vertical foliation, hence they have bounded Hausdorff distance
(by Masur \cite{Mas80} if the vertical foliation is uniquely ergodic, which happens almost surely; 
for a more general result see also Ivanov \cite{Iv}), so $\gamma'$ tracks 
the random walk sublinearly and passes through $x$. 
Let us finally remark that the vertical foliation of $\gamma$ will coincide almost surely with $F$: 
indeed, if we call $F_0$ the vertical foliation of $\gamma$, $F_0$ will be almost surely uniquely ergodic, 
hence $\lim_{n \to \infty} \gamma(An) = F_0$. Thus, by sublinear tracking and (\cite{KM}, Lemma 1.4.2), 
$(w_n x)$ also tends to $F_0$ as $n \to \infty$, so $F = F_0$.
\end{proof}

\subsection{Hadamard spaces} \label{hadamard}

Another nice class of spaces with non-positive curvature geometry 
are Hadamard spaces. 

\begin{definition}
A metric space $(X, d)$ is called a \emph{Hadamard space} if it is simply connected, complete, geodesic and CAT(0).
\end{definition}

An example of a Hadamard space is a simply connected, complete Riemannian manifold of nonpositive sectional curvature.
An equivalent characterization is the following, given by Bruhat and Tits (see \cite{Ba} for general references):

\begin{proposition}
Let $(X, d)$ be a complete metric space. $X$ is a \emph{Hadamard space} if and only if for every $x, y \in X$, there exists a point $m \in X$ such that
$$d^2(z, m) \leq \frac{1}{2} (d^2(z, x) + d^2(z, y)) - \frac{1}{4} d^2(x, y) \qquad \textup{for all }z \in X,$$ 
where we used $d^2(x, y)$ as a shorthand for $(d(x, y))^2$.
\end{proposition}

Let us suppose from now on that $X$ is a locally compact Hadamard space.
A boundary of $X$ is given by the set $X(\infty)$ of equivalence classes of geodesic rays: 
$$X(\infty) := \{ \sigma : [0, \infty) \to X \ \textup{geodesic ray} \} / \sim$$
where $\sigma_1 \sim \sigma_2$ iff there exists $M$ such that $d(\sigma_1(t), \sigma_2(t)) \leq M$ for all $t \geq 0$.
It is not hard to verify that $X(\infty)$ is compact, and isometries of $X$ extend to homeomorphisms of the boundary 
(see \cite{EO}, or \cite{Ba}).

In order to apply our technique, we need to be able to connect with a geodesic almost every 
pair of points on the boundary; this is not true in all Hadamard spaces (for instance in the euclidean plane), 
hence we need to impose a slightly stronger negative-curvature condition.  
For instance, Eberlein and O'Neill \cite{EO} introduced the concept of \emph{visibility manifolds}, which are precisely 
Hadamard manifolds for which our stable visibility condition holds (\cite{EO}, Proposition 4.4), so we get sublinear tracking in this case. 
However, in the following we will see that the argument works for a slightly more general class of Hadamard spaces. 

\begin{definition}
A geodesic $\sigma : \mathbb{R} \to X$ is \emph{regular} if it does not bound a flat half plane.
A locally compact Hadamard space is said to be of \emph{rank one} if there is at least one regular geodesic. 
\end{definition}

For instance, the hyperbolic plane  $\mathbb{H}^2$ has rank one, while the euclidean plane $\mathbb{R}^2$ is still a Hadamard space, 
but not of rank one.


We also need to ensure that our group $G$ of isometries is large enough. The following condition was introduced by Eberlein and Chen \cite{EC}:

\begin{definition}
We say that $\xi$, $\eta \in X(\infty)$ are \emph{dual} if there is a sequence $g_n \in G$ such that, for some $x \in X$,  $g_n x\to \xi$ 
and $g_n^{-1} x \to \eta$. We say that $G$ satisfies the \emph{duality condition} if for any geodesic $\sigma$ the endpoints $\sigma(-\infty)$ 
and $\sigma(\infty)$ are dual.
\end{definition}




Let us remark that if $X$ is a Hadamard manifold, and $G$ acts properly discontinuously in such a way that 
the quotient $X/G$ has finite volume, then the duality condition is satisfied.

By applying the previous techniques we have the following




\begin{theorem}
Let $X$ be a locally compact Hadamard space of rank one such that $X(\infty)$ contains at least three points, and $G$ be a countable group of isometries of $X$ 
satisfying the duality condition. Let $x \in X$ be a basepoint, and $\mu$ a probability measure on $G$ whose support generates $G$ as a semigroup, 
and with finite first moment.
Then there exists some $A \geq 0$ such that for $\mathbb{P}$-a.e. sample path, there exists a geodesic ray 
$\gamma : [0, \infty) \to X$ with $\gamma(0) = x$ and 
$$\lim_{n \to \infty} \frac{d(w_n x, \gamma(An))}{n} = 0.$$
\end{theorem}

\begin{proof}
By (\cite{Ba}, Thm. III.4.11), both $\mathbb{P}$-a.e. sample path and $\check{\mathbb{P}}$-a.e. sample backward path converge to a point on $\partial X$,
and the harmonic measure has full support on $X(\infty)$ since the Dirichlet problem is solvable.
If we now consider the set $\mathcal{R}:= \{ (\xi, \eta) \in X(\infty) \times X(\infty) \ : \ \exists \sigma 
\textup{ regular geodesic with }\sigma(-\infty)= \xi, \sigma(\infty) = \eta \}$
of endpoints of regular geodesics, it has full harmonic measure. 
Indeed, $\nu \otimes \check{\nu}(\mathcal{R}) > 0$ since $\mathcal{R}$ is open  (\cite{Ba}, Lemma III.3.1) and harmonic measure has full support, hence by ergodicity of the shift map
and the fact that $\mathcal{R}$ is $G$-invariant one has $\nu \otimes \check{\nu}(\mathcal{R}) = 1$.
Moreover, given any pair $(\xi, \eta) \in \mathcal{R}$, there exists a unique geodesic $\sigma$ with $\sigma(-\infty) = \xi$, $\sigma(\infty) = \eta$, 
(\cite{Ba}, Corollary I.5.8]) hence we have a measurable map $G^\mathbb{Z} \to \Gamma X$ defined on a set of full measure
and we can apply Theorem \ref{mainabs}.
\end{proof}

A different proof of sublinear tracking for general Hadamard spaces  
has been given by Karlsson and Margulis \cite{KarMar};
our technique gives an independent proof, even though with some restrictions.
An identification of the Poisson boundary for cocompact groups acting on rank one Hadamard manifolds 
is contained in \cite{BL}, where bi-infinite paths are already used. See also \cite{Ka00}.

\subsection{Lamplighter groups over trees} \label{ll}

We shall now turn to random walks on certain wreath products; in particular, we shall call 
\emph{lamplighter group over a tree} a wreath product of the form 
$$G := \mathbb{Z}_m \wr F_k,$$ 
where $\mathbb{Z}_m$ is a cyclic group of order $m \geq 2$ and $F_k$ is a free (non-abelian) group of rank $k \geq 2$.
Our goal is to establish sublinear tracking for the action of $G$ on its Cayley graph, endowed 
with a word metric.

Let us first recall a few general definitions.
Let $K$ and $H$ be two groups. Let us denote as $fun(H, K)$ the set of finitely supported 
functions from $H$ to $K$, i.e. the direct sum
$$fun(H, K) := \bigoplus_{h \in H} K.$$
The group $H$ acts on $fun(H, K)$ by translations; indeed, for each $h$ in $H$ 
we define $T_h : fun(H, K) \to fun(H, K)$ as 
$$T_h(f)(x) := f(h^{-1}x) \qquad \forall x \in H, f \in fun(H, K).$$
Let us recall that the \emph{(restricted) wreath product} $K \wr H$ is defined to 
be the semidirect product 
$$K \wr H := H \rtimes fun(H, K).$$

The reason for the name ``lamplighter'' is the following. Under the canonical choice of generators, the Cayley graph of $F_k$ is a homogeneous 
tree  of degree $2k$, the set of vertices of which we shall denote $\mathbb{T}$; we shall also denote as $e$ the vertex of the tree corresponding to the identity element.
We think of a function $f : \mathbb{T} \to \mathbb{Z}_m$ as determining the status of 
``lights'' which are located on each vertex of $\mathbb{T}$ and can assume different levels of ``color'' or ``intensity''
from $0$ to $m-1$; for this reason, each $f : \mathbb{T} \to \mathbb{Z}_m$ will be called a \emph{configuration}. The \emph{support} of a configuration $f$ is the set
$$\textup{supp }f := \{ x \in \mathbb{T} \ : \ f(x) \neq 0 \}.$$
Thus, the set $\mathcal{C} := fun(\mathbb{T}, \mathbb{Z}_m)$ will be called the set of \emph{finitely-supported configurations}, 
while the closure of $\mathcal{C}$ in the topology of pointwise convergence is the set $\overline{\mathcal{C}}$ of all (possibly 
infinitely-supported) configurations.

In this section, we shall consider the action of $G$ on its Cayley graph $X$.
The set of vertices of $X$ is the product $\mathbb{T} \times \mathcal{C}$, hence each vertex of $X$ is represented 
by a pair $(x, f)$ where $x$ is a vertex of the tree, which we think of as the current 
position of the lamplighter person, and $f : \mathbb{T} \to \mathbb{Z}_m$ is a finite configuration 
of lights.
We shall consider the word metric on the Cayley graph with respect to the standard generating set
$$S := \bigcup_{r = 1}^{k} (a_r^{\pm}, 0) \cup (e, \pm \delta_e)$$
where $a_1, \dots, a_k$ are free generators of $F_k$; moreover, $0$ denotes the configuration where all lights 
are off, and $\delta_e$ is the configuration whose value is $1$ on $e$ and $0$ everywhere else.
With this choice, 
there is an edge between $(x, f)$ and $(x',f')$ if either $f = f'$ and there is an edge between $x$ and $x'$ in the tree, 
or $x = x'$ and the functions $f$ and $f'$ differ only at $x$, i.e. 
$f(y) = f'(y)$ for all $y \neq x$, and $|f(x) - f'(x)| = 1$. 
The main result of this section is the following.

\begin{theorem} \label{llst}
Let $G = \mathbb{Z}_m \wr F_k$ be a lamplighter group over a tree, acting on its Cayley graph $X$
with the word metric described above. 
Let $\mu$ be a probability measure on $G$
with finite first moment and such that the support of $\mu$ does not fix any finite set of ends of $F_k$.
Then there exists $A > 0$ such that for each $x \in G$ and almost every sample path 
$(w_n)$ of the random walk determined by $\mu$ there exists a geodesic ray $\gamma : [0, \infty) \to G$ such that 
$$\lim_{n \to \infty}\frac{d(w_n x , \gamma(An))}{n} = 0.$$
\end{theorem}

Random walks on lamplighter groups have been widely studied, especially because they provided interesting counterexamples
to several conjectures, most notably by Kaimanovich and Vershik \cite{KaiVer} and then Erschler \cite{Ers}.
Lamplighter random walks over trees have been introduced by Karlsson and Woess \cite{KarWoe}, who identified 
the Poisson boundary with the set of limit configurations (see below).
For a more complete bibliography, we refer to the references within the abovementioned works.
Note that the theorem actually works in slightly greater generality, replacing the free group $F_k$ with an arbitrary 
group of isometries of an infinite, locally finite tree, and replacing the cyclic group $\mathbb{Z}_m$ 
with any finite group.
Finally, note that taking $k = 1$ we get the lamplighter random walk over $\mathbb{Z}$, which is one of the main 
examples in \cite{KaiVer}; in that case, our technique works as long as the drift of the projected walk on $\mathbb{Z}$
is non-zero (otherwise, the Poisson boundary is trivial and we do not have convergence to the boundary).

In order to prove the theorem, let us analyze the geometry of $X$ in more detail.
Any two vertices $x, y \in \mathbb{T}$ are connected by a unique geodesic, which we will denote $[x, y]$.
Moreover, for each vertex $x$ of $\mathbb{T}$ we can define the set 
$$U_x := \{ z \in \mathbb{T} \ : \ x \in [e, z] \}$$
of points which are ``further away'' from the origin than $x$. 
Note that if $U_x \cap U_y = \emptyset$, then each continuous path from a point 
in $U_x$ to a point in $U_y$ contains all edges of the geodesic $[x, y]$.

The geodesics in $X$ have a simple and well-known interpretation. Namely, let $u = (x, f)$ and $v = (y, g)$
be two vertices of $X$. Then the projection to the tree of a geodesic in $X$ between $u$ and $v$ corresponds to a 
shortest ``travelling salesman'' path which starts from $x$, reaches all points of the set 
$\textup{supp }(f-g)\ $
and ends in $y$. 
In the case of trees, the shortest travelling salesman path can be completely characterized \cite{Parry}. 
In particular, we shall need the following property.

\begin{lemma} \label{crossonce}
Let $x, y$ be two vertices of the tree $\mathbb{T}$, and $u = (x, f), v = (y, g)$ two elements 
of $X$, and $\gamma$ a geodesic in $X$ joining $u$ to $v$. 
Then the projection of $\gamma$ to $\mathbb{T}$ contains every edge of the geodesic $[x, y]$ exactly once. 
\end{lemma}

\begin{proof} 
Let $\widetilde{\gamma}$ be the projection of the path $\gamma$ to the tree. 
Since the tree contains no loops, then each continuous path in it from $x$ to $y$ 
must contain each edge of the geodesic $[x, y]$ at least once, and so does $\widetilde{\gamma}$.
Suppose now by contradiction that there exists an edge $E \subseteq [x, y]$ such that  
$\widetilde{\gamma}$ runs along $E$ more than one time. Then, if we call $z$ and $w$ the endpoints of $E$, 
we have that $\widetilde{\gamma}$ is of the form $\widetilde{\gamma} = \gamma_1 \cup E \cup \gamma_2 \cup E^{-1} \cup \gamma_3 \cup E \cup \gamma_4$, 
where $\gamma_1$ is a path which joins $x$ to $z$, $\gamma_2$ joins $w$ to $w$, $\gamma_3$ joins $z$ to $z$ and $\gamma_4$ joins 
$w$ to $y$, and $E^{-1}$ means the edge $E$ traversed in the opposite direction. This is clearly not a shortest travelling salesman path, because the path 
$\gamma' := \gamma_1 \cup \gamma_3 \cup E \cup \gamma_2 \cup \gamma_4$ visits the same locations and is shorter, hence the claim 
is proven.
\end{proof}



Let us now define a boundary for the graph $X$. 
Let $\overline{\mathbb{T}} := \mathbb{T} \cup \partial \mathbb{T}$ be the end compactification of the tree:
note that $X$ has only one end, even though $\mathbb{T}$ has infinitely many.
We shall take as a boundary for $X$ the set of \emph{limit configurations}, consisting of pairs of 
one end $\xi$ of $\mathbb{T}$ and a configuration of lights whose support can only 
accumulate at $\xi$:
$$\partial X := \{ (\xi, f) \in \partial \mathbb{T} \times \overline{\mathcal{C}} \ : \ 
(\textup{supp }f) \setminus U \textup{ is finite } \forall U \in Nbd(\xi)  \}$$
where $Nbd(\xi)$ denotes the set of neighbourhoods of $\xi$ in $\overline{\mathbb{T}}$.
Note that this boundary is not compact, indeed its closure is the whole $\partial \mathbb{T} \times \overline{ \mathcal{C}}$.
The set of limit configurations has been first proposed as a boundary for lamplighter groups over $\mathbb{Z}^k$ by 
Kaimanovich and Vershik \cite{KaiVer}, and for lamplighters over trees by Karlsson and Woess \cite{KarWoe}.

In order to apply the techniques of the previous sections, we need a version of the stable visibility 
property for $\partial X$. Unfortunately, the boundary $\partial X$ is not stably visible 
in the sense of Definition \ref{svdef}; however, we shall show that a suitable subset of $\partial X \times \partial X$ 
satisfies a version of stable visibility, and that such subset has full measure for the (doubly-infinite) random walk. 

Indeed, we say that a subset $\Lambda \subseteq \partial X \times \partial X$ 
is \emph{stably visible} if for any pair $(\xi, \eta) \in \Lambda$ and any sequence $\gamma_n = [\xi_n, \eta_n]$ of geodesic segments in $X$ 
with $\xi_n \to \xi$ and $\eta_n \to \eta$, there exists a bounded set in $X$ which intersects all $\gamma_n$.

\begin{proposition} \label{llsv}
The subset $\Lambda \subseteq \partial X \times \partial X$ defined as 
$$\Lambda := \{ ((x, f), (y, g)) \in \partial X \times \partial X \ : \ x \neq y \}$$
is stably visible.
\end{proposition}

\begin{proof}
Let $u_n = (x_n, f_n)$ and $v_n = (y_n, g_n)$ two sequences of vertices of $X$ such that $u_n \to u_\infty = (x_\infty, f_\infty)$ 
and $v_n \to v_\infty = (y_\infty, g_\infty)$ tend to two points on the boundary of $X$, with $x_\infty \neq y_\infty$.
For each $n$, let $\gamma_n$ be a geodesic in $X$ joining $u_n$ to $v_n$.
By the definition of $\partial X$, 
we can choose a neighbourhood $U$ of $x_\infty$ which does not intersect the support of $g_\infty$, and 
a neighbourhood $V$ of $y_\infty$ which does not intersect the support of $f_\infty$, and in such a way that $U \cap V = \emptyset$.
Without loss of generality, we can assume $U = U_x$ and $V = U_y$ for some $x, y \in \mathbb{T}$.
For $n$ large enough, we have $u_n \equiv u_\infty$ on the complement of $U$, and 
$v_n \equiv v_\infty$ on the complement of $V$, and $x_n \in U, y_n \in V$. 
By Lemma \ref{crossonce}, the projection $\widetilde{\gamma}_n$ of $\gamma_n$ 
to the tree contains each edge of the geodesic $[x, y]$ exactly once.
Thus, if $z$ is a vertex on $[x, y]$, there exists an element $w_n$ on the geodesic $\gamma_n$
which is of the form $w_n = (z, h_n)$.
Moreover, when the path $\widetilde{\gamma}_n$ reaches $z$, then the lights in $U$ have been already 
switched, while none of the lights in $V$ has been switched. Thus, the support of $h_n$ is contained in $\mathbb{T} \setminus U \cup V$.
As a consequence, since on $\mathbb{T} \setminus U \cup V$ we have $f_n \equiv f_\infty$ and $g_n \equiv g_\infty$, we get the inclusion
$$\textup{supp }h_n \subseteq \textup{supp }f_\infty \cup \textup{supp }g_\infty \setminus U \cup V$$
so the support of $h_n$ is contained in a ball around the origin of radius independent of $n$.
Thus, both coordinates of $w_n$ are bounded independently of $n$ and the claim is proven.
\end{proof}

\begin{proof}[Proof of Theorem \ref{llst}.]
Karlsson and Woess prove (\cite{KarWoe}, Theorem 2.9) that almost every sample path converges to a point in $\partial X$ (using the 
result about convergence to the ends for random walks on trees in \cite{CS}). 
Moreover, they also prove that the harmonic measure is non-atomic on the set of ends, 
so the set 
$$\Lambda := \{ ((x, f), (y, g)) \in \partial X \times \partial X \ : \ x \neq y \}$$
has full $\nu \otimes \hat{\nu}$-measure. Thus, by Proposition \ref{llsv}, the set $\Lambda$
is stably visible and by the exact same argument as in the proof of Lemma \ref{measurable} 
the function 
$$D(\xi, \eta) := \sup_{\gamma \in P(\xi, \eta)} d(e, \gamma)$$
is measurable and almost everywhere finite. The sublinear tracking property then follows by Theorem \ref{mainabs}.
Finally, the group generated by the support of $\mu$ is non-amenable, since it does not fix a finite set of ends of the tree; 
thus, the Poisson boundary of the walk is non-trivial (see e.g. \cite{KaiVer}), hence by the entropy criterion (Theorem \ref{entrocrit}) one has $A >0$ 
as claimed.
 \end{proof}






\begin{thebibliography}{99}


\bibitem{Ba}
\textsc{Ballmann, W.},
\textit{Lectures on spaces of nonpositive curvature},
DMV Seminar {\bf 25}, Birkh\"auser Verlag, Basel, 1995.

\bibitem{BH}
\textsc{Bridson, M. R.} and \textsc{Haefliger, A.},
\textit{Metric spaces of non-positive curvature},
Grundlehren der Mathematischen Wissenschaften \textbf{319}, Springer-Verlag, Berlin, 1999.

\bibitem{BL}
\textsc{Ballmann, W.} and \textsc{Ledrappier, F.},
\textit{The {P}oisson boundary for rank one manifolds and their
              cocompact lattices},
Forum Math. {\bf 6} (1994), 301--313.

\bibitem{CS}
\textsc{Cartwright, D. I.} and \textsc{Soardi, P. M.},
\textit{Convergence to ends for random walks on the automorphism group
              of a tree},
Proc. Amer. Math. Soc. {\bf 107} (1989), no. 3, 817--823.

\bibitem{EC}
\textsc{Chen, S. S.} and \textsc{Eberlein, P.},
\textit{Isometry groups of simply connected manifolds of nonpositive
              curvature},
Illinois J. Math. {\bf 24} (1980), 73--103.

\bibitem{De}
\textsc{Derriennic, Y.},
\textit{Quelques applications du th\'eor\`eme ergodique sous-additif},
Conference on {R}andom {W}alks ({K}leebach, 1979) ({F}rench),
Ast\'erisque {\bf 74}, 183--201, Soc. Math. France, Paris, 1980.


\bibitem{Du}
\textsc{Duchin, M.}, 
\textit{Thin triangles and a multiplicative ergodic theorem for Teichmüller geometry}, 
preprint arXiv:math.GT/0508046.

\bibitem{EO}
\textsc{Eberlein, P.} and \textsc{O'Neill, B.},
\textit{Visibility manifolds},
Pacific J. Math. {\bf 46} (1973), 45--109.

\bibitem{Ers}
\textsc{Erschler, A.},
\textit{Liouville property for groups and manifolds},
Invent. Math. {\bf 155} (2004), no. 1, 55--80.

\bibitem{Fl}
\textsc{Floyd, W. J.},
\textit{Group completions and limit sets of {K}leinian groups},
Invent. Math. {\bf 57} (1980), 205--218.

\bibitem{FK}
\textsc{Furstenberg, H.} and \textsc{Kesten, H.},
\textit{Products of random matrices},
Ann. Math. Statist. {\bf 31} (1960), 457--469.

\bibitem{Ge}
\textsc{Gerasimov, V.}, 
\textit{Floyd maps to the boundaries of relatively hyperbolic groups},
 arXiv:1001.4482 [math.GR]. 

\bibitem{Hopf}
\textsc{Hopf, H.},
\textit{Enden offener {R}\"aume und unendliche diskontinuierliche
              {G}ruppen},
Comment. Math. Helv. {\bf 16} (1944), 81--100.

\bibitem{Iv} 
\textsc{Ivanov, N. V.},
\textit{Isometries of {T}eichm\"uller spaces from the point of view of
              {M}ostow rigidity},
in \textit{Topology, ergodic theory, real algebraic geometry},
Amer. Math. Soc. Transl. Ser. 2 {\bf 202}, 131--149, Amer. Math. Soc., Providence, 2001.

\bibitem{Ka85}
\textsc{Kaimanovich, V. A.},
\textit{An entropy criterion of maximality for the boundary of random
              walks on discrete groups},
Dokl. Akad. Nauk SSSR {\bf 280} (1985), 1051--1054.

\bibitem{Ka87}
\textsc{Kaimanovich, V. A.},
\textit{Lyapunov exponents, symmetric spaces and a multiplicative
              ergodic theorem for semisimple {L}ie groups},
Zap. Nauchn. Sem. Leningrad. Otdel. Mat. Inst. Steklov. (LOMI) {\bf 164} (1987), 29--46, 196--197.

\bibitem{Ka94}
\textsc{Kaimanovich, V. A.},
\textit{The {P}oisson boundary of hyperbolic groups},
C. R. Acad. Sci. Paris S\'er. I Math. {\bf 318} (1994), 59--64.


\bibitem{Ka00}
\textsc{Kaimanovich, V. A.},
\textit{The {P}oisson formula for groups with hyperbolic properties},
Ann. of Math. (2) {\bf 152} (2000), 659--692.

\bibitem{KM}
\textsc{Kaimanovich, V. A.} and \textsc{Masur, H.},
\textit{The {P}oisson boundary of the mapping class group},
Invent. Math. {\bf 125} (1996), 221--264.     

\bibitem{KaiVer}
\textsc{Kaimanovich, V. A.} and \textsc{Vershik, A. M.},
\textit{Random walks on discrete groups: boundary and entropy},
Ann. Probab. {\bf 11} (1983), no. 3, 457--490.

\bibitem{Kar03}
\textsc{Karlsson, A.}, 
\textit{Boundaries and random walks on finitely generated infinite groups}, 
Ark. Mat. {\bf 41} (2003), 295--306.

\bibitem{KL}
\textsc{Karlsson, A.} and \textsc{Ledrappier, F.},
\textit{On laws of large numbers for random walks},
Ann. Probab. {\bf 34} (2006), 1693--1706.

\bibitem{KarMar}
\textsc{Karlsson, A.} and \textsc{Margulis, G.},
\textit{A multiplicative ergodic theorem and nonpositively curved
spaces}, Commun. Math. Phys. {\bf 208} (1999), 107–123.

\bibitem{KarWoe} 
\textsc{Karlsson, A.} and \textsc{Woess, W.},
\textit{The {P}oisson boundary of lamplighter random walks on trees},
Geom. Dedicata {\bf 124} (2007), 95--107.

\bibitem{Le}
\textsc{Ledrappier, F.},
\textit{Some asymptotic properties of random walks on free groups},
Topics in probability and {L}ie groups: boundary theory, CRM Proc. Lecture Notes {\bf 28},
117--152, Amer. Math. Soc., Providence, RI, 2001.

\bibitem{Mas75}
\textsc{Masur, H.},
\textit{On a class of geodesics in {T}eichm\"uller space},
Ann. of Math. (2) {\bf 102} (1975), 205--221.

\bibitem{Mas80}
\textsc{Masur, H.},
\textit{Uniquely ergodic quadratic differentials},
Comment. Math. Helv. {\bf 55} (1980), 255--266.

\bibitem{Mas09}
\textsc{Masur, H.},
\textit{Geometry of {T}eichm\"uller space with the {T}eichm\"uller
              metric},
in Surveys in differential geometry. {V}ol. {XIV}. {G}eometry of
              {R}iemann surfaces and their moduli spaces,
Surv. Differ. Geom. {\bf 14}, 295--313, Int. Press, Somerville, MA, 2009.
		
\bibitem{MW}
\textsc{Masur, H.} and \textsc{Wolf, M.},
\textit{Teichm\"uller space is not {G}romov hyperbolic},
Ann. Acad. Sci. Fenn. Ser. A I Math. {\bf 20} (1995), 259--267.

\bibitem{Mi}
\textsc{Minsky, Y.}, 
\textit{Extremal length estimates and product regions in
              {T}eichm\"uller space},
Duke Math. J. {\bf 83} (1996), 249--286.

\bibitem{MM}
\textsc{Mj, M.},
\textit{Cannon-Thurston maps for Kleinian groups}, 
arXiv:1002.0996 [math.GT].


\bibitem{O}
\textsc{Oseledec, V. I.},
\textit{A multiplicative ergodic theorem. {C}haracteristic {L}japunov,
              exponents of dynamical systems},
Trudy Moskov. Mat. Ob\v s\v c. {\bf 19} (1968), 179--210.

\bibitem{Parry}
\textsc{Parry, W.},
\textit{Growth series of some wreath products},
Trans. Amer. Math. Soc. {\bf 331} (1992), no. 2, 751--759.

\bibitem{Woe93}
\textsc{Woess, W.},
\textit{Fixed sets and free subgroups of groups acting on metric spaces},
Math. Z. {\bf 214} (1993), 425--439.

\end{thebibliography}
\end{document}